\documentclass[12pt]{article}
\usepackage[cm]{fullpage}
\usepackage[small]{titlesec}
\usepackage[cm]{fullpage}

\usepackage{amsmath,amscd, float,rotating}
\usepackage{pb-diagram}

\usepackage{latexsym}
\usepackage{amsfonts}
\usepackage{amsmath}
\usepackage{amssymb}

\numberwithin{equation}{section}
\newcommand{\qed}{\hfill \ensuremath{\Box}}

\def\XXint#1#2#3{{\setbox0=\hbox{$#1{#2#3}{\int}$}
\vcenter{\hbox{$#2#3$}}\kern-.5\wd0}}

\newcommand{\dbar}{\overline{\partial}}

\newcommand{\ddt}[1]{\frac{\partial #1}{\partial t}}

\newcommand{\ddbar}{\frac{\sqrt{-1}}{2\pi} \partial\dbar}

\begin{document}
\newcounter{remark}
\newcounter{theor}
\setcounter{remark}{0} \setcounter{theor}{1}
\newtheorem{claim}{Claim}
\newtheorem{theorem}{Theorem}[section]
\newtheorem{proposition}{Proposition}[section]
\newtheorem{lemma}{Lemma}[section]
\newtheorem{definition}{Definition}[section]
\newtheorem{conjecture}{Conjecture}[section]
\newtheorem{corollary}{Corollary}[section]
\newenvironment{proof}[1][Proof]{\begin{trivlist}
\item[\hskip \labelsep {\bfseries #1}]}{\end{trivlist}}
\newenvironment{remark}[1][Remark]{\addtocounter{remark}{1} \begin{trivlist}
\item[\hskip \labelsep {\bfseries #1
\thesection.\theremark}]}{\end{trivlist}}
\newenvironment{example}[1][Example]{\addtocounter{remark}{1} \begin{trivlist}
\item[\hskip \labelsep {\bfseries #1
\thesection.\theremark}]}{\end{trivlist}}
~

\begin{center}
{\large \bf
Degeneration of K\"ahler-Ricci solitons on Fano manifolds
\footnote{Work supported in part by
National Science Foundation grants DMS-0757372, DMS-0847524 and DMS-0905873.}}
\bigskip\bigskip

{D.H. Phong*, Jian Song$^\dagger$, and Jacob Sturm$^\ddagger$} \\

\bigskip
{}
\smallskip

\end{center}

\medskip

\begin{abstract}

We consider the space $\mathcal{KR}(n, F)$ of  K\"ahler-Ricci solitons on $n$-dimensional Fano manifolds with Futaki invariant bounded by $F$. We prove a partial $C^0$ estimate for $\mathcal{KR}(n,F)$ as a generalization of the recent work of Donaldson-Sun for Fano K\"ahler-Einstein manifolds. In particular, any sequence in $\mathcal{KR}(n,F)$ has a convergent subsequence in the Gromov-Hausdorff topology to a K\"ahler-Ricci soliton on a $\mathbb{Q}$-Fano variety with log terminal singularities.

\end{abstract}

\section{Introduction}

Let $X$ be a Fano manifold admitting a smooth K\"ahler-Ricci soliton, that is a metric $g_{i\bar j}$ satisfying the equation
$$
Ric(g) = g + L_V g.
$$
where $V$ is a holomorphic vector field, and $L_V$ is the Lie derivative along $V$.
The holomorphic vector field can be expressed in terms of the Ricci potential $u$, with
\begin{equation}
\label{soliton}
R_{i\bar j}= g_{i\bar j} -  u_{i\bar j}, ~~u_{ij}=  u_{\bar i \bar j}= 0, ~V^i = -g^{i\bar j}u_{\bar j}. 
\end{equation}
The Futaki invariant associated to $(X, g, V)$ is given by
$$
\mathcal{F}_X(V) = \int_X |\nabla u|^2 dV_g = \int_X |V|^2 dV_g \geq 0. 
$$

\begin{definition}
Let $\mathcal{KR}(n, F)$ be the set of K\"ahler-Ricci solitons $(X, g)$ with $$\dim X = n, ~Ric(g) = g + L_V g,~  \mathcal{F}_X (V) \leq F.$$
\end{definition}

The main result is a partial $C^0$ estimate for K\"ahler-Ricci solitons.  Let $(X, g) \in \mathcal{KR}(n, F)$ and $\omega_g$ be the K\"ahler form for $g$. Let $h$ be a hermitian metric on $K_X^{-1}$ with $Ric(h) = \omega_g$, which is unique up to a multiplicative normalization. We define the $L^2$-inner product on $H^0(X, K_X^{-k})$ by 
$$\langle s, s' \rangle = k^n \int_X |s|_{h^k} ^2 \omega_g^n$$ for any $s, s' \in H^0(X, K_X^{-k})$.
Let $\{s_j\}_{j=1}^{N_k}$ be an orthonormal basis in $H^0(X, K_X^{-k})$ with respect to $\langle , \rangle$. Then the Bergman kernel $\rho_{X, k}$ is defined to be 
\begin{equation}
\rho_{X, k} = \sum_j |s_j|^2_{h^k}.
\end{equation}
The Bergman kernel $\rho_{X,k}$ is independent of the normalization of $h$. The partial $C^0$-estimate introduced and proved for smooth Fano surfaces with K\"ahler-Einstein metrics in \cite{T1}, involves a uniform lower bound for the Bergman kernel $\rho_{X, k}$.

\begin{theorem} \label{parc0} There exist $k(n, F) \in \mathbb{Z}^+$ and $\epsilon(n, F)>0$ such that for any $(X, g)  \in \mathcal{KR}(n, F)$, the Bergman kernel $\rho_{X,k} $ of $H^0(X, K_X^{-k})$ is uniformly bounded below by $\epsilon$, i.e., 
\begin{equation}
\inf_{z\in X} \rho_{X,k} (z) \geq \epsilon.
\end{equation}

\end{theorem}

The proof of Theorem \ref{parc0} relies on the arguments in \cite{DS} and \cite{Z, TZ}. A consequence of Theorem \ref{parc0} is the following compactness result, which is obtained by a suitable modification of the argument in \cite{DS}. 

\begin{theorem} \label{comp} Any sequence $(X_i, g_i) \in \mathcal{KR}(n, F)$, after passing to a subsequence, 
 converges in the Gromov-Hausdorff topology to a compact metric length space $(X_\infty, g_\infty)$ satisfying:

\begin{enumerate}

\item $X_\infty$ is a projective $\mathbb{Q}$-Fano variety with log terminal singularities. The singular set $\Sigma_{X_\infty}$ of $X_\infty$  is a subvariety of $X_\infty$ with  complex codimension no less than $2$;

\item $g_\infty$ is a K\"ahler current on $X_\infty$ with bounded local potentials.
Furthermore  $g_\infty$ is smooth on $X_\infty \setminus \Sigma_{X_\infty}$, and satisfyies the K\"ahler-Ricci soliton equation %
\begin{equation}
Ric(g_\infty) = g_\infty + L_{V_\infty} g_\infty,
\end{equation}
where $V_\infty$ is a holomorphic vector field on $X_\infty$.  The upper bound of $\|V_\infty\|_{L^\infty(X_\infty, g_\infty)}$ only depends on $n$ and $F$;

\item $g_i$ converges to $g_\infty$ in $C^\infty$ on $X_\infty \setminus \Sigma_{X_\infty}$.

\end{enumerate}

\end{theorem}

The metric completion of $X_\infty\setminus \Sigma_{X_\infty}$ by $g_\infty$ coincides with the Gromov-Hausdorff limit $X_\infty$ itself as a compact metric length space, and the limiting holomorhpic vector field $V_\infty$ extends globally to $X_\infty$ since $X_\infty$ is normal. We also remark that $g_\infty$ is  bounded below by a multiple of the Fubini-Study metric by applying estimates similar to Schwarz lemma. We can now obtain a compactification of $\mathcal{KR}(n, F)$ in the Gromov-Hausdorff topology. 
\begin{definition}

Let $\overline{\mathcal{KR}(n, F)}$ be the closure of $\mathcal{KR}(n, F)$ defined by the set of all K\"ahler-Ricci soliton $(X_\infty, g_\infty)$ such that there exists a convergent sequence $(X_i, g_i) \in \mathcal{KR}(n, F)$ with $(X_\infty, g_\infty)$ being the limit in Theorem \ref{comp}. 

\end{definition}

Theorem \ref{comp} also implies certain algebraic boundedness for $\overline{\mathcal{KR}(n, F)}$.

\begin{corollary} There exist $m=m(n, F)\in \mathbb{Z}^+$, $C=C(n, F)>0$ and $\delta=\delta(n, F)>0$, such that for any $X\in \overline{\mathcal{KR}(n, F)}$, 
\begin{equation}
-mK_X~ is~Cartier, ~~ [K_X]^n\leq C, ~~discr(X)> -1+ \delta.
\end{equation}
Here $discr(X)$ is the discrepancy of $X$, defined by the equation (\ref{discrepancy}) below.

\end{corollary}

Finally, we raise two natural questions closely related to the main results. 
\medskip

\noindent $\bullet$ Does there exist $F=F(n)>0$ such that for any K\"ahler-Ricci soliton $g_{i\bar j}$ on an $n$-dimensional  Fano manifold $X$ and $V$
the corresponding holomorphic vector field, the Futaki invariant is uniformly bounded by
\begin{equation}\label{fbound}
 \mathcal{F}_X(V) \leq F ? 
\end{equation} 
If this holds, the compactness result will hold for all K\"ahler-Ricci solitons on $n$-dimensional Fano manifolds.

 In general, (\ref{fbound}) does not hold for the space of K\"ahler-Ricci solitons on  Fano varieties  with log terminal singularities. For example, we can consider a weighted projective surface $X_m$ defined by the polytope $P_m$ as the convex hull of three points $(-1,-1)$, $(2/m, -1)$ and $(-1, m+1) $. The discrepancy  
of $X_m$ is given by $-1 + 2/m$. Hence  $discr(X_m)$ tends to $-1$ and $c_1(X_m)^2$ tends to $\infty$ as $m\rightarrow \infty$. There always exists a smooth orbifold K\"ahler-Ricci soliton $(g_m, V_m)$ on $X_m$ by \cite{WZ} and the Futaki invariant of $X_m$ tends to $\infty$ as $m\rightarrow \infty$. The compactness for singular K\"ahler-Ricci solitons on $\mathbb{Q}$-Fano varieties might still hold with bounds such as the Futaki invariant, $c_1^n$ and the discrepancy of the singularities. This seems to suggest that the Futaki invariant for K\"ahler-Ricci solitons are related to the boundedness problem for Fano varieties in birational geometry.  

\medskip

\noindent $\bullet$ For any $(X, g)\in \overline{\mathcal{KR}(n, F)}$, is the Ricci curvature of $g$ uniformly bounded on the regular part of $X$? This is equivalent to saying that the potential of the holomorphic vector field $V$ is a quasi-plurisubharmonic function with respect to a multiple of $g$.

\section{Geometric  estimates }

Since smooth Fano manifolds with fixed dimension can only have finitely many deformation types \cite{Na, KMM}, the intersection number  $[K_X]^n$ is uniformly bounded. 

\begin{lemma} For any $n>0$, there exists $c=c(n)>0$ such that for any Fano manifold $X$, 
\begin{equation}
c^{-1} \leq c_1^n (X)  \leq c. 
\end{equation}

\end{lemma}

We consider Perelman's entropy functional \cite{P} for a Fano manifold $(X, g)$ with the associated K\"ahler form $\omega_g \in c_1(X)$, which is defined by 
\begin{equation}
\mathcal{W}(g, f) =\frac{1}{V} \int_X (R + |\nabla f|^2 + f - n)  e^{-f} dV_g,
\end{equation}
where $V= c_1^n(X)$.
The $\mu$-functional is defined by 
\begin{equation}
\mu(g) = \inf_f\left\{ \mathcal{W}(g, f)~\left|~ \frac{1}{V} \int_X e^{-f} d V_g = 1 \right. \right\}. 
\end{equation}

For compact gradient shrinking solitons, we have the following well-known identities (cf. \cite{Cao})
\begin{equation}\label{soleq1}
R+\Delta u = n,
\end{equation}
\begin{equation}\label{soleq2}
R+|\nabla u|^2 = u + constant.
\end{equation}
In the case of K\"ahler-Ricci solitons, we have 
$$Ric(g)= g - \ddbar u,  ~u_{ij} = u_{\bar i \bar j} =0.$$
From now on, we always assume the following normalizing condition for $u$
\begin{equation}
\frac{1}{V} \int_X e^{-u} dV_g = 1, ~~V= c_1^n(X).
\end{equation}
Integrating (\ref{soleq1}) against $e^{-u}$, one can determine the constant in (\ref{soleq2})  after an integration by parts, 
\begin{equation}\label{soleqn3}
R+|\nabla u|^2 = u - \frac{1}{V} \int_X u e^{-u}dV_g + n.
\end{equation}

The following lemma is due to Tian-Zhang \cite{TZ}. Since the proof is short,
we  include it here for the convenience of the reader.

\begin{lemma}   There exists $A=A(n, F)>0$ such that for any $(X, g) \in \mathcal{KR}(n, F)$, 

\begin{equation}
\mu(g) \geq -A.
\end{equation}

\end{lemma}

\begin{proof} Straightforward calculations using (\ref{soleqn3}) show that
\begin{equation}
\mu(g)= W(g, u) = \frac{1}{V} \int_X (R+|\nabla u|^2+u - n) e^{-u} dV_g
=-\frac{1}{V} \int_X u e^{-u} dV_g.
\end{equation}
It then suffices to show that $\frac{1}{V} \int_X u e^{-u} dV_g$ is uniformly bounded below.  
By (\ref{soleqn3}), 
$$\int_X u e^{-u} dV_g \geq   \int_X u dV_g - F  $$
or equivalently,
\begin{eqnarray*}
\int_X u e^{-u} dV_g &\geq&   \int_X 2u dV_g -  \int_X u e^{-u} dV_g- 2F  \\
&=& \int_{u\leq -1} u( 2 - e^{-u}) dV_g + \int_{u\geq -1}u (2 - e^{-u}) dV_g - 2F\\
&\geq& -( 2 + \max_{x \geq -1}xe^{-x} )V - 2F.
\end{eqnarray*} 

\qed
\end{proof}

The following lemma is well-known and due to Ivey \cite{I}. 
\begin{lemma} The scalar curvature $R$ is positive for all compact shrinking gradient solitons.
\end{lemma}

Then following Perelman's argument (see \cite{ST})
 combined with the above two lemmas,  one obtains the following lemma.
\begin{proposition} 
\label{proposition}
There exists $C=C(n, F) >0$ such that for all $(X, g) \in \mathcal{KR}(n, F)$,

\begin{equation}\label{geoqu}
|u| + |\nabla u|_g^2 + |R(g)| + Diam_g(X)  \leq C.
\end{equation}

\end{proposition}

\begin{proof} We give a sketch of the proof. The K\"ahler-Ricci soliton can be considered as a solution of the K\"ahler-Ricci flow
$$\ddt{g(t)}= -Ric(g(t)) +g(t), ~~~g(0) = g$$ after applying the holomorphic vector field (cf. \cite{PSSW}). Let $u(t)$ be the Ricci potential of $g(t)$ defined by $$Ric(g(t)) = g(t) - \ddbar u(t), ~~\int_X e^{-u(t) }dV_{g(t)} = V.$$
Since $R(t)>0$, the volume of $g$ is uniformly bounded and $\mu(g(t))=\mu(g)$ is uniformly bounded below, following \cite{ST}, $\int_X u(t) e^{-u(t)} dV_{g(t)}$ is uniformly bounded and $u(t)$ is uniformly bounded below. Notice that for any continuous function $h(z, t) = F(u(t) , |\nabla u(t) |_{g(t)}, \Delta_{g(t)} u(t))$, 
$$\max_{z\in X} h(z, t) = \max_{z\in X} h(z, 0).$$ Hence using Perelman's argument of the maximum principle, one has uniform bounds for $$ \frac{|\nabla u(t)|_{g(t)}|}{ u(t) +1 - \min_{z} u(t)}, ~~\frac{-\Delta_{g(t)} u(t)}{ u(t) +1 - \min_{z} u(t)} .$$ This will lead to the uniform bound of the diameter of $g(t)$ using the uniform lower bound of $\mu(g(t))$. The proposition then easily follows.

\qed
\end{proof}

\section{Conformal transformation and analytic compactness }

The following is the idea of  Z. Zhang \cite{Z}.  Let $(X, g)\in \mathcal{KR}(n, F)$, we apply a conformal transformation using the Ricci potential $u$
\begin{equation}
\tilde g= e^{-\frac{1}{n-1} u} g. 
\end{equation}
Then the uniform bounds on $u$ 
and on $|\nabla u|_{g}$
imply that $\tilde g$ and $g$ are $C^1$ equivalent.

The Ricci curvatures of the metrics $\tilde g$ and $g$ are related by the well-known equation
(see e.g. \cite{A}, section 6.1)
\begin{equation}
\tilde R_{ij}
=
R_{ij}+\nabla_i\nabla_j u+{1\over 2(n-1)}\nabla_iu\nabla_ju
-
{1\over 2(n-1)}(|\nabla u|_g^2-\Delta u)g_{ij}
\end{equation}
It follows that from the soliton equation
(\ref{soliton})
and Proposition \ref{proposition} that
the Ricci curvature of $\tilde g$ is bounded:

\begin{lemma}\label{riccibd} There exists $C=C(n, F)$ such that for any $(M, g) \in \mathcal{KR}(n, F)$,

$$- C \tilde g \leq Ric( \tilde g)  \leq C \tilde g.$$

\end{lemma}
With Lemma \ref{riccibd}, one can apply the general compactness results as in \cite{CC1, CC2, CC3, CCT}. The uniform bound of $u$  implies that the diameter of $(X, \tilde g)$ is uniformly bounded above and the volume of $(X, \tilde g)$ is uniformly bounded on both sides. In addition, one has the uniform nonlocal collapsing property for $\tilde g$. All the constants only depend on $n$ and $F$.
 We also have the following volume comparison:
 
\begin{corollary} There exist $\kappa=\kappa(n, F)>0$ such that for any $(X, g) \in \mathcal{KR}(n, F)$, 
\begin{equation}
\kappa^{-1} r^{2n} \leq Vol(B_g(z, r)) \leq \kappa r^{2n},
\end{equation} 
for any $z\in X$ and $r\leq 1$. 

\end{corollary}

One can now apply the results of Cheeger-Colding to $\tilde g$. With a careful treatment for the tangent cones, one derives the following theorem \cite{TZ, Z}, making use of  the uniform $C^1$ equivalence between  $g$ and $\tilde g$.

\begin{theorem}\label{TZ} Let $(X_i, g_i)\in \mathcal{KR}(n, F)$ 
be a sequence in $\mathcal{KR}(n,F)$
with uniformly bounded volumes.
Then after passing to a subsequence
if necessary, the sequence $(X_i,g_i)$ converges in the Gromov-Hausdorff sense to a compact metric length space $(X_\infty, g_\infty)$ satisfying the following:

\begin{enumerate}

\item The singular set $\Sigma_{X_\infty}$ of $X_\infty$ is of codimension no less than $4$;

\item On $X_\infty\setminus \Sigma_{X_\infty}$;
$g_\infty$ is a smooth K\"ahler metric satisfying the K\"ahler-Ricci soliton equation. The metric completion of $(X_\infty \setminus \Sigma_{X_\infty}, g_\infty)$ coincides with $(X_\infty, g_\infty)$;

\item $g_i$ converges to $g_\infty$ in $C^\infty$ topology on $X_\infty \setminus \Sigma_{X_\infty}$.

\end{enumerate}

\end{theorem}

The $C^\infty$ convergence on the regular part of $X_\infty$ is achieved by making use of a variant of Perelman's pseudolocality theorem due to \cite{FZZ} since the soliton metric is a solution of the Ricci flow. The goal of the rest of the paper is to show that $X_\infty$ is isomorphic to a projective variety equipped with a canonical K\"ahler-Ricci soliton metric. 

\section{$L^2$-estimates}

In this section, we will obtain some uniform $L^2$-estimates for $H^0(X, K_X^{-k})$ when $X\in \mathcal{KR}(n, F)$. Using the same notations in \cite{DS}, we denote  
$$ K_X^\sharp = K_X^{-k}, ~h^\sharp = h^k, ~\omega^\sharp = k\omega, ~L^{p, \sharp}(X)= L^p(X, \omega^\sharp), $$ 
where $h$ is the hermitian metric on $K_X^{-1}$ with its curvature $Ric(h)=\omega$. The hermitian metric on $K_X^{-1}$ is equivalent to a volume form on $X$ and since $g$ satisfies the soliton equation, we can normalize $h$ such that 
$$h = e^{-u} \omega^n, ~~\int_X e^{-u} \omega^n = \int_X\omega^n = c_1(X)^n. $$ 
We also note that the Bergman kernal $\rho_{X, k}$ is invariant under any scaling for $h$. 

Since the Sobolev constant is uniform for $\tilde g$, so it is for $g$ as $g$ and $\tilde g$ are uniformly equivalent, when $(X, g) \in \mathcal{KR}(n, F)$. 
The following proposition, which shows that
Proposition \ref{proposition} in \cite {DS} can be extended to the case of K\"ahler-Ricci solitons, is one of the key components in the proof of Theorem 1.1:

\begin{proposition} \label{l2est} There exist  $a=a (n, F)$, $K_1=K_1(n, F)$, $K_2=K_2(n, F) >0$ such that if $(X, g) \in \mathcal{KR}(n, F)$ and $s\in H^0(X, K_X^{-k})$ for $k\geq 1$, then 
\begin{enumerate}

\item $\|s\|_{L^{\infty, \sharp}} \leq K_1 \|s\|_{L^{2, \sharp}}$;

\item  $ \|\nabla s\|_{L^{\infty, \sharp}} \leq K_2 \|s\|_{L^{2, \sharp}} $;

\item We consider the $L^2$ inner product for any $K_X^{-k}$-valued $(0,1)$-form $\sigma$ defined by
$$\int_X |\sigma|_{h^\sharp, g^\sharp} ^2 e^{-u} dV_{g^\sharp} $$ 
and its induced adjoint operator $\dbar _u^*$ of $\dbar$. Then the Beltrami-Laplace operator $\Delta_{\dbar, u}^\sharp = \dbar \dbar^*_u + \dbar_u^* \dbar $ is invertible with
\begin{equation}
\Delta_{\dbar, u}^\sharp \geq a. 
\end{equation}

\end{enumerate}

\end{proposition}

\begin{proof} The proof proceeds in a similar way as in \cite{DS}.  

\begin{enumerate} 

\item  Let $(X, g)$ be any element in $\mathcal{KR}(n, F)$. The bound
on the Sobolev constant of $(X, g)$ only depends on $n$ and $F$, and so does the Sobolev constant for the rescaled metric $(X, kg, h^k)$. For simplicity, we write $|s|$ for $|s|_{h^k}$ and $s\in H^0(X,K_X^{-k})$. A pointwise calculation shows that
$$ \Delta |s| \geq - |s|. $$ 
The first inequality then follows immediately from Moser iteration.

\item We drop the index $\sharp$ for simplicity. A direct calculation give the following identities. 
\begin{eqnarray}
\Delta |\nabla s|^2 
&=& \langle \nabla s, \nabla s \rangle_{Ric(\omega)} + |\nabla\nabla s|^2 - 2|\nabla s|^2 + n|s|^2
\\
\Delta |\nabla s|
&=& - |\nabla s| + \frac{n}{2} \frac{|s|^2}{|\nabla s|} +\frac{\langle \nabla s, \nabla s \rangle_{Ric(\omega)}}{2|\nabla s|} + \frac{|\nabla \nabla s|^2}{2|\nabla s|} - \frac{|\nabla |\nabla s|^2 |^2}{4 |\nabla s|^3}
\nonumber
\end{eqnarray}
On the other hand, we have the following inequality
\begin{eqnarray}
|\nabla |\nabla s|^2 |^2 
&=& g^{k\bar\ell}g^{i\bar j}g^{p\bar q}
(\nabla_i\nabla_k s\nabla_{\bar q}\nabla_{\bar \ell}\bar s\nabla_ps\nabla_{\bar j}\bar s
+
g_{p\bar\ell}\nabla_i\nabla_k s\nabla_{\bar j}\nabla_{\bar q}\bar s
\nonumber
\\
&&
\qquad\qquad\quad
+g_{k\bar j}\nabla_{\bar q}\nabla_{\bar\ell}\bar s
\nabla_i s\nabla_p s
+
g_{k\bar j}g_{p\bar\ell}|s|^2
\nabla_is\nabla_{\bar q}s)\nonumber
\\
&\leq&|\nabla\nabla s|^2 |\nabla s|^2 + 2 |s||\nabla s|^2 |\nabla\nabla s| + |s|^2|\nabla s|^2 
\end{eqnarray}
Using this inequality, we obtain  
\begin{eqnarray*}
\Delta |\nabla s| &\geq& -  |\nabla s| + \frac{n}{2} \frac{|s|^2}{|\nabla s|} +\frac{\langle \nabla s, \nabla s \rangle_{Ric(\omega)}}{2|\nabla s|} \\
&& + \frac{|\nabla \nabla s|^2}{4|\nabla s|} - \frac{|s||\nabla\nabla s|}{2 |\nabla s|} - \frac{|s|^2}{4 |\nabla s|} \\
&=& - \frac{2 - k^{-1} }{2} |\nabla s| + \frac{n}{2} \frac{|s|^2}{|\nabla s|} -\frac{\langle \nabla s, \nabla s \rangle_{\ddbar u }}{2|\nabla s|} \\
&& + \frac{|\nabla \nabla s|^2}{4|\nabla s|} - \frac{|s||\nabla\nabla s|}{2 |\nabla s|} - \frac{|s|^2}{4 |\nabla s|} \\
\end{eqnarray*}
We can now start the Moser iteration process, 
\begin{eqnarray*}
&&- \int |\nabla s |^p \Delta |\nabla s| 
= \frac{4p}{(p+1)^2}\int |\nabla  |\nabla s|^{(p+1)/2} |^2\\
&\leq&2^{-1} \int  ( 2 |\nabla s|^{p+1}  - n |s|^2|\nabla s|^{p-1} - 2^{-1} |\nabla\nabla s|^2 |\nabla s|^{p-1} + |s| |\nabla \nabla s| |\nabla s|^{p-1}  + 2^{-1} |s|^2 |\nabla s|^{p-1} ) \\
&& + 2^{-1}\int    \langle \nabla s, \nabla s \rangle_{\ddbar u} |\nabla s|^{p-1}. 
\end{eqnarray*}
We will deal with the last quantity using integration by part. Suppose $h= e^{-\varphi}$. In normal coordinates, $\varphi_i = \varphi_{\bar j} = 0$, $\varphi_{i\bar j} = g_{i\bar j} =\delta_{i\bar j}.$ 
$$ ( |\nabla s|^2)_{\bar j} = g^{k\bar l} (-\varphi_{k\bar j} s \bar s _{\bar l} + s_k \bar s _{\bar l \bar j}) =- s\bar s_{\bar j} + s_k \bar s_{\bar k \bar j}. $$
Then 
\begin{eqnarray*}
&&\left|\int \langle \nabla s, \nabla s \rangle_{\ddbar u}  | \nabla s |^{p-1} \right|\\
&=& \left| \int (s_i - s \varphi_i) (\bar s_{\bar j} - \bar s \varphi_{\bar j} ) u_{i\bar j} e^{-\varphi} |\nabla s|^{p-1} \right|\\
&=&\left| \int u_{\bar j} ( s_{ii} \bar s_{\bar j} - s s_i \varphi_{i\bar j}) |\nabla s|^{p-1} + \frac{p-1}{2} \int  u_{\bar j} s_i \bar s_{\bar j} (|\nabla s|^2)_{\bar j} |\nabla s|^{p-3} \right|\\
&\leq & Cp \int ( |\nabla\nabla s| |\nabla s|^{p} + |s| |\nabla s|^p+ |\nabla s|^{p+1}). 
\end{eqnarray*}
Without loss of generality, we can assume that $\|s\|_{L^2} =1$. Then by applying the Cauchy-Schwarz inequality and (1)  in the proposition, we have 
\begin{eqnarray*}
\int |\nabla ( |\nabla  s|^{(p+1)/2})  |^2 
&\leq& Cp^3 \int ( |\nabla s|^{p+1} + |s|^2|\nabla s|^{p-1} )\\
&\leq& CK_2p^3 \int ( |\nabla s|^{p+1} + \|s\|_{L^2}^2 |\nabla s |^{p-1})\\
&\leq & CK_1 p^3 \max ( \|\nabla s\|_{L^{p+1}}^{p+1}, 1)
\end{eqnarray*}
From the Sobolev inequality, we have, with $\beta= \frac{n}{n-1}$
$$ \|\nabla s \|_{L^{p\beta}} \leq (CK_1p)^{3/p} \max( \|\nabla s\|_{L^p}, 1). $$
Then Moser's iteration applies starting from $p=2$ and we obtain 

$$\|\nabla s\|_{L^\infty} \leq C.$$

\item 

The last inequality follows from the Bochner-Kodaira-Nakano identity, where the weight $e^{-u}$ eliminates the $\ddbar u$ in the soliton equation for $Ric(g)$.  More precisely, let 
$$\langle \sigma , \sigma \rangle = \int_X |\sigma|^2 e^{-u} (\omega^\sharp)^n$$ be the $L^2$- product for 
$\sigma\in \Omega^{0,1}\otimes K_X^{-k},$
$$\Delta_{\dbar, u} = \dbar\dbar^*_u + \dbar^*_u \dbar, ~~\Delta_{D_u} = -(D_u)^{\bar j} (D_u)_{\bar j}, $$
where $(D_u)$ is the covariant derivative on $\Omega^{(0, 1)}\otimes K_X^{-k}$ with respect to the K\"ahler metric $g$ and the hermitian metric $h e^{-u}$. 
We have the following Bochner-Kodaira-Nakano identity (cf. \cite{KM}). 
\begin{eqnarray*}
(\Delta_{\dbar, u}  \sigma)_{\bar j}  &=&  (\Delta_{D_u} \sigma)_{\bar j} + (g)^{i\bar q} \left(g_{i\bar j} + u_{i\bar j} +(Ric(g))_{i\bar j} \right) \sigma_{\bar q} \\
&=&  (\Delta_{D_u} \sigma)_{\bar j} + \frac{k+1}{k} \sigma_{\bar j}.
\end{eqnarray*}
This immediately implies that 
\begin{equation}
\langle \Delta_{\dbar, u} \sigma, \sigma \rangle \geq  e^{-\sup u} ||\sigma||_{L^{2}}^2.
\end{equation}

\end{enumerate}
The proof of the proposition is complete.
\qed
\end{proof}

\section{Partial $C^0$ estimate}

We now consider a slight modification of the $H$-property introduced by Donaldson-Sun \cite{DS}. 

\begin{definition} We consider the follow data $(p_*, D, U, \Lambda, J, g, h, A)$ satisfying 

\begin{enumerate}

\item  $(p_*, U, J, g)$  is an open bounded K\"ahler manifold with a complex structure $J$, a K\"ahler metric $g$ and a base point $p_* \in U$; 

\item $\Lambda \rightarrow U$ is a hermitian line bundle equipped with a hermitian metric $h$. $A$ is the connection induced by the hermitian metric $h$ on $\Lambda$, with  curvature $\Omega(A) = g$. $D$ is an open disc with $p_*  \in D\subset \subset U$.

\end{enumerate}

\noindent The data $(p_*, D, U, \Lambda, J, g, h, A)$ is said to have the $H'$-property if there exist $C>0$ and  a compactly supported smooth section $\sigma: U\rightarrow \Lambda$   satisfying 

\begin{enumerate}

\item [$H'_1$:]   $\|\sigma \|_{L^2} < (2\pi)^{n/2}$;

\item [$H'_2$:]  $|\sigma(p_*) | >3/4$;

\item [$H'_3$:] for any holomorphic section $\tau$ of $\Lambda$ over a neighborhood of $\overline{D}$,
$$|\tau(p_*)| \leq C \| \tau \|_{L^2(D)}  ;$$

\item [$H'_4$:]  $\|\dbar \sigma \|_{L^2} < \min \left( \frac{a^{1/2}}{4 C}, \frac{(2\pi)^{n/2}}{10\sqrt{2}} \right), $ where $a=a(n, F)$ is the constant in Proposition \ref{l2est}; 

\item [$H'_5$:]  $\sigma$ is constant in $D$.

\end{enumerate}

\end{definition}

It is straightforward to check that the $H'$-property is open with respect to $C^l$ variations in $(g, J, A)$ for any $l \geq 0$ with $(p_*, D, U, \Lambda)$ being fixed. 

The standard application of $L^2$-estimate implies the following lemma (cf. \cite{DS}).

\begin{lemma} Suppose $(X, g)\in \mathcal{KR}(n, F)$. There exists $b=b(n, F)>0$ such that if  $p\in D \subset\subset U\subset X$ satisfies property $H$ with  $\Lambda= K_X^{-k}$ for some $k>0$, then 
\begin{equation}
\rho_{X, k}(p) > b.
\end{equation}

\end{lemma}

\begin{proof} Let $\sigma$ be a smooth section in the definition of $H'$-property. 
We define
$$\tau= \dbar_u^* (\Delta_{\dbar, u})^{-1} \dbar \sigma, ~~ s= \sigma - \tau.$$
$\dbar \tau = \dbar  \sigma$ since $\dbar \Delta_{\dbar, u} = \Delta_{\dbar, u} \dbar $.  Therefore $\dbar s  = 0$ and so $s\in H^0(X, K_X^{-k})$. 
The $L^2$ norm of $s$ is bounded by 
\begin{eqnarray*}
||s||_{L^2} & \leq &  ||\sigma||_{L^2} + ||\tau||_{L^2} \\
&\leq & (2\pi)^{n/2} + a^{-1/2} ||\dbar \sigma ||_{L^2}\\
&\leq & (2\pi)^{n/2}(1+(200a)^{-1/2} ).
\end{eqnarray*}
On the other hand, by $H'_3$ and the calculations above,
\begin{eqnarray*}
|s|(p) &\geq& |\sigma|(p) - |\tau|(p)\\
&> & 3/4 - C \| \tau |_{L^2(D)}\\
&\geq& 3/4 - C a^{-1/2} \|\dbar \sigma\|_{L^2} \\
&>& 1/2.
\end{eqnarray*}
The lemma then immediately follows.

\qed
\end{proof}

\bigskip

\noindent{\it Proof of Theorem \ref{parc0}. }  For any sequence $(X_i, g_i) \in \mathcal{KR}(n, F)$, a subsequence will
converge to some $(X_\infty, g_\infty)$ as in Theorem \ref{TZ}. For any $p\in X_\infty$, any tangent cone $C(Y)$ at $p$ is a metric cone whose link $Y$ has a singular set $\Sigma_{Y}$ of real codimension at least $4$. The cone metric on $C(Y)$ is a smooth Ricci flat K\"ahler metric $g_{C(Y)}$ on the regular part with  
\begin{equation}
g_{C(Y)}= dr^2 + r^2g_Y = \ddbar r^2/2,
\end{equation} 
where $r=|z|$ is the distance from the vertex $p$ to $z$. With the $L^2$ estimates proved in Proposition \ref{l2est},  the argument of Donaldson-Sun in section 3 of \cite{DS} can be immediately applied. In particular, there exist $k>0$ and $(p_*,D, U)$ with $p_*\in D\subset U\subset\subset C(Y)_{reg}$, such that $(p_*, D, U, \Lambda^k, J_{C(Y)}, kg_{C(Y)}, A^{\otimes k}_{C(Y)})$ satisfies Property $H'$ for sufficiently small perturbations of  $g_{C(Y)}$ and $J_{C(Y)}$ in $C^0(U)$, where
 $\Lambda\rightarrow U$ is a trivial hermitian line bundle with hermitian metric $h_{C(Y)}= e^{-|z|^2/2}$.
 
 \qed

\section{Limiting K\"ahler-Ricci solitons}

\begin{definition}
\label{discrepancy}
Let $X$ be a normal variety with $K_X$ being a $\mathbb{Q}$-Cartier divisor. Let $\pi: \tilde X \rightarrow X$ be a resolution of singularities with 
\begin{equation}
K_{\tilde X} = \pi^* K_X + \sum a_i E_i,
\end{equation}
where $a_i \in \mathbb{Q}$ and $E_i$ are the exceptional divisors of $\pi$. 
$X$ is said to have log terminal singularities if $a_i >-1$, for all $i$. The discrepancy of $X$ is defined by 
\begin{equation}
discr(X)= \inf_{\pi, i} (1, a_i),
\end{equation}
for all resolution $\pi: \tilde X \rightarrow X$.

\end{definition}

After establishing the partial $C^0$ estimate in Theorem \ref{parc0}, the arguments in sections 4.1, 4.2 and 4.3 of \cite{DS} can be faithfully applied to show that there exists $k=k(n, F)>0$ such that any sequence $(X_i, g_i, K_{X_i}^{-k})\in \mathcal{KR}(n, F)$, after passing to a subsequence, converges to a polarized limit $(X_\infty, g_\infty, K_{X_\infty}^{-k})$ with $X_\infty = Proj(R(X_\infty, K_{X_\infty}^{-k}))$ being a normal projective variety, where $R(X_\infty, K_{X_\infty}^{-k})= \oplus_m H^0(X_\infty, K_{X_\infty}^{-mk})$. Without loss of generality, we can embed $X_i$ and $X_\infty$ in a fixed $\mathbb{P}^{N_k}$ using the $L^2$-orthonormal basis $ \{s^{(i)}_j \}_{j=0}^{N_k}$ of  $H^0(X_i, K_{X_i}^{-k})$ and $ \{s^{(\infty)}_j \}_{j=0}^{N_k}$ of $H^0(X_\infty, K_{X_\infty}^{-k})$ respectively. Let $\rho_{X_i, k}= \sum |s^{(i)}_j|^2_{h_i^k}$ and $\rho_{X_\infty, k}=\sum_j |s^{(\infty)}_j|^2_{h_\infty^k}$ be the Bergman kernals.

\begin{proposition} \label{logt} $X_\infty$ is a projective $\mathbb{Q}$-Fano variety with log terminal singularities. In particular, the algebraic singular set coincides with the singular set of $g_\infty$.

\end{proposition}

\begin{proof} By Proposition \ref{l2est} and Theorem \ref{parc0}, $\log \rho_{X_\infty, k}$ is uniformly bounded. For any point $p\in X_\infty$, there exists a holomorphic section $s\in H^0(X_\infty, K_{X_\infty}^{-k})$ such that $s$ does not vanish on an open $U$ neighborhood with $\inf_U |s|\geq \epsilon$. Then $\Theta_s= (s\wedge \bar s)^{-1/k}$ is a volume measure and 
\begin{equation}
\int_{X_\infty \cap U} (s \wedge \bar s)^{-1/k} = \int_{X_\infty \cap U} |s|^{-2/k} dV_{g_{\infty}} \leq (\epsilon)^{-2/k} V.
\end{equation}
Let $\pi: \tilde X \rightarrow X_\infty$ be a resolution of singularities. Then $\pi^* \Theta_s$ is $L^1$-integrable on $\tilde X$, and since $\Theta_s$ can have only algebraic singularities, $\pi^*\Theta_s$ is $L^{1+\epsilon}$-integrable on $\tilde X$ for some $\epsilon>0$. This implies that $X_\infty$ has at worst log terminal singularities.

\qed
\end{proof}

\begin{proposition} The limiting variety $(X_\infty, g_\infty)$ arising from Proposition \ref{logt} solves the K\"ahler-Ricci soliton on $X_\infty$ in the following sense. 

\begin{enumerate}

\item $g_\infty$ is a global K\"ahler current on $X_\infty$ with bounded local K\"ahler potentials. 

\item  $g_\infty$ solves the K\"ahler-Ricci soliton equation on  $X_\infty^{reg}$
\begin{equation}
Ric(g_\infty)+ \nabla^2 u_\infty= g_\infty ,
\end{equation}
 for some smooth real valued potential function $u_\infty$  on $X_\infty^{reg}$. 

\item $\|u_\infty\|_{C^1(X_\infty^{reg})}< \infty$, and thus the holomorphic vector field $V_\infty= \uparrow \dbar u_\infty$ extends to a global holomorphic vector field on $X_\infty$ with $\|V_\infty\|_{L^\infty(X_\infty, g_\infty)}< \infty$. In particular, the Futaki invariant of $(X_\infty, g_\infty)$
can be bounded by $F$,
$$\mathcal{F}_{X_\infty}(V_\infty) =\int_{X_\infty} |V_\infty|^2 dV_{g_\infty} \leq F.$$

\end{enumerate}

\end{proposition}

\begin{proof}  We first prove that the local K\"ahler potentials of $g_\infty$ are uniformly bounded. Let  $$\omega_{FS, i} = k^{-1} \ddbar \log \sum_j s_j^{(i)}\wedge\overline{s_j^{(i)}}, ~\omega_{FS, \infty} = k^{-1} \ddbar \log \sum_j s_j^{(\infty)}\wedge \overline{s_j^{(\infty)}}$$
be the Fubini-Study metrics from the embeddings by $\{s_j^{(i)} \}_j$ and $\{s_j^{(\infty)}\}_j$.  Then 
$$\omega_{g_i} = \omega_{FS, i} + \ddbar \log \rho_{X_i, k}, ~\omega_{g_\infty} = \omega_{FS, \infty} + \ddbar \log \rho_{X_\infty, k}.$$
By Proposition \ref{l2est}, $\rho_{X_i, k}$ and $\rho_{X_\infty, k}$ are uniformly bounded in $L^\infty$. By the partial $C^0$ estimate, $\rho_{X_i, k}$ and $\rho_{X_\infty, k}$ are uniformly bounded below away from $0$. Therefore $\varphi_i = \log \rho_{k,i}$ and $\varphi_\infty = \log \rho_{X_\infty, k}$ are uniformly bounded in $L^\infty$. 

Note that the hermitian metric on $K_{X_i}^{-1} $ and $K_{X_\infty}^{-1}$ are given by $h_i= e^{-u_i}\omega_{g_i}^n$ and $h_\infty= e^{-u_\infty} \omega_{g_\infty}^n$. Since $u_i$ and $|\nabla u_i|_{g_i}$ are uniformly bounded, $u_i$ converges in $C^\alpha$ on $X_\infty^{reg}$ to $u_\infty$. From the smooth convergence of $g_i$ to $g_\infty$ on $X_\infty^{reg}$, $u_i$ converges in $C^\infty$ to $u_\infty$ on $X_\infty^{reg}$ with $|u_\infty|$ and $|\nabla u_\infty|$ uniformly bounded on $X_\infty^{reg}$. Furthermore, $g_\infty$ satisfies the soliton equation 
on $X_\infty^{reg}$
$$Ric(g_\infty) = g_\infty - \nabla^2 u_\infty= g_\infty+ L_{V_\infty} g_\infty,$$
where $(V_\infty)^i = -(g_\infty)^{i \bar j} (u_\infty)_{\bar j}$ is the holomorphic vector field on $X_\infty^{reg}$ induced by $u_\infty$. Since $X_\infty$ is normal, $V_\infty$ extends to a bounded global holomorphic vector field on $X_\infty$ with $||X_\infty||_{L^\infty(X_\infty, g_\infty)}<\infty$  and $\mathcal{F}_{X_\infty}(V_\infty) \leq F$. 
In fact, if we let $\Omega_{FS, \infty}$ be the smooth volume form on $X_\infty $ with $\ddbar \log \Omega_{FS, \infty} = - \omega_{FS, \infty}$, then $\varphi_{ \infty}$ satisfies a global Monge-Amp\`ere equation on $X_\infty$
\begin{equation}
(\omega_{FS, \infty} + \ddbar \varphi_{ \infty} )^n = e^{ - \varphi_\infty  + u_\infty} \Omega_{FS, \infty} \end{equation}
and on $X_\infty^{reg}$ (cf. \cite{Y1, EGZ}).

\qed
\end{proof}

\bigskip

\noindent {\bf{Acknowledgements:}} The authors would like to thank Xiaowei Wang, Ved Datar and Bin Guo for many valuable discussions.

\bigskip

{\noindent \footnotesize $^*$ Department of Mathematics\\
Columbia University, New York, NY 10027\\

\noindent $\dagger$ Department of Mathematics\\
Rutgers University, Piscataway, NJ 08854\\

\noindent $\ddagger$ Department of Mathematics\\
Rutgers University, Newark, NJ 07102\\ }

\end{document}